\documentclass[12pt]{amsart}

\usepackage[colorlinks]{hyperref}

\usepackage{fullpage}
\usepackage{stmaryrd}
\usepackage{graphicx}
\usepackage{hyperref}
\usepackage{amsmath,amssymb,amsfonts,mathrsfs,amscd}
\usepackage{newtxtext,newtxmath}
\usepackage[all,cmtip]{xy}

\newtheorem{theorem}{Theorem}[section]
\newtheorem{lem}[theorem]{Lemma}
\newtheorem{cor}[theorem]{Corollary}

\newtheorem{con}[theorem]{Conjecture}

\newtheorem{thmA}{Theorem}

\newtheorem{corA}[thmA]{Corollary}

\theoremstyle{definition}
\newtheorem{defn}[theorem]{Definition}

\theoremstyle{remark}

\numberwithin{equation}{section}

\DeclareMathOperator{\Irr}{Irr}
\DeclareMathOperator{\IBr}{IBr}

\DeclareMathOperator{\Hall}{Hall}

\DeclareMathOperator{\Ker}{Ker}

\DeclareMathOperator{\Lin}{Lin}

\def\X{{\rm X}_{\pi}}
\def\Y{{\rm X}_{\pi'}}

\def\I{{\rm I}_\pi}
\def\B{{\rm B}_\pi}

\def\LPS{L^*_\varphi}
\def\LP{L_\varphi}
\def\VI{{\rm I}_\varphi}
\def\LN{\tilde{L}_\varphi}

\def\HJ#1{\par\medskip\noindent{\bf#1.}\bgroup\it \ }
\def\EHJ{\egroup\medskip}

\newcommand{\NM}{\vartriangleleft}

\begin{document}

\title{Lifts of Brauer characters in characteristic two}

\author{Ping Jin*}
\address{School of Mathematical Sciences, Shanxi University, Taiyuan, 030006, China.}
\email{jinping@sxu.edu.cn}

\author{Lei Wang}
\address{School of Mathematical Sciences, Shanxi University, Taiyuan, 030006, China.}
\email{wanglei0115@163.com}


\thanks{*Corresponding author}


\date{}

\begin{abstract}
A conjecture raised by Cossey in 2007 asserts that if $G$ is a finite $p$-solvable group and $\varphi$ is an irreducible $p$-Brauer character of $G$ with vertex $Q$,
then the number of lifts of $\varphi$ is at most $|Q:Q'|$.
This conjecture is now known to be true in several situations for $p$ odd,
but there has been little progress for $p$ even.
The main obstacle appeared in characteristic two is that
all the vertex pairs of a lift are neither linear nor conjugate.
In this paper we show that if $\chi$ is a lift of an irreducible $2$-Brauer character in a solvable group, then $\chi$ has a linear Navarro vertex if and only if all the vertex pairs of $\chi$ are linear,
and in that case all of the twisted vertices of $\chi$ are conjugate.
Our result can also be used to study other lifting problems of Brauer characters in characteristic two. As an application, we prove a weaker form of Cossey's conjecture for $p=2$ ``one vertex at a time".
\end{abstract}

\maketitle

\section{Introduction}
Fix a prime number $p$, and let $G$ be a finite $p$-solvable group.
The Fong-Swan theorem states that for each irreducible $p$-Brauer character $\varphi\in\IBr_p(G)$, there exists some irreducible complex character $\chi\in\Irr(G)$ such that $\chi^0=\varphi$, where $\chi^0$ denotes the restriction of $\chi$ to
the $p$-regular elements of $G$.
Such a character $\chi$ is called a {\bf lift} of $\varphi$.
In \cite{C2007}, Cossey proposed the following conjecture,
which also appears as Problem 3.6 of \cite{N2023}.

\begin{con}[Cossey]
Let $G$ be a $p$-solvable group, and let $\varphi\in\IBr_p(G)$.
Then $$|\LP|\leq|Q:Q'|,$$
where $\LP$ is the set of all lifts of $\varphi$ and $Q$ is a vertex for $\varphi$
(in the sense of Green).
\end{con}

This global/local conjecture seems difficult to prove,
although some progress has been made for $p$ odd.
In 2007, Cossey \cite{C2007} verified his conjecture for (solvable) groups of odd order,
and in 2011, Cossey, Lewis and Navarro \cite{CLN2011} proved the conjecture under the conditions that either $Q$ is abelian or $Q\NM G$ and $p\neq2$.
Also, Cossey and Lewis \cite{CL2010a} computed the exact number of lifts of $\varphi$
in the case where $\varphi$ lies in a block with a cyclic defect group.
For further background material of this conjecture, the reader is referred to the survey paper of Cossey \cite{C2010}.

In \cite{C2008}, Cossey assigned to each $\chi\in\Irr(G)$ a pair $(Q,\delta)$
consisting of a $p$-subgroup $Q$ of $G$ and a character $\delta\in\Irr(Q)$,
which is called a {\bf vertex pair} for $\chi$ (for precise definition, see Section 2 below or Definition 4.5 of \cite{C2010}).
Furthermore, the pair $(Q,\delta)$ is said to be {\bf linear} if $\delta\in\Lin(Q)$ is a linear character.
Of particular note is the fact that if $\chi$ is a lift of $\varphi\in\IBr_p(G)$ and if $(Q,\delta)$ is a linear vertex pair for $\chi$, then $Q$ is necessarily a vertex for $\varphi$ (see Theorem 4.6(d) of \cite{C2010}).
The following result, which is a main theorem of \cite{CL2012},
plays a key role in the study of Cossey's conjecture as well as many other lifting problems about Brauer characters;
see, for example, \cite{C2008, C2009, C2010, CL2011, CLN2011, C2012, CL2012, L2010, WJ2022, WJ2023}.

\begin{lem}[Cossey-Lewis]\label{odd}
Let $G$ be a $p$-solvable group with $p>2$, and suppose that $\chi\in\Irr(G)$ is a lift of some $\varphi\in\IBr_p(G)$. Then all of the vertex pairs for $\chi$ are linear and conjugate in $G$.
\end{lem}

Following the notation of \cite{CL2012},
we write $\LP(Q,\delta)$ for the set of those lifts of $\varphi\in\IBr_p(G)$ with vertex pair $(Q,\delta)$ and let $N_G(Q,\delta)$ denote the stabilizer of $\delta$ in $N_G(Q)$,
where $Q$ is a $p$-subgroup of a $p$-solvable group $G$ and $\delta\in\Irr(Q)$.
In the situation of Lemma \ref{odd}, let $Q$ be a vertex for $\varphi$,
so that each lift of $\varphi$ has vertex pair $(Q,\delta)$ for some $\delta\in\Lin(Q)$,
and let $\{\delta_1,\ldots,\delta_n\}$ be a set of representatives of the $N_G(Q)$-orbits in $\Lin(Q)$ (under the natural action of $N_G(Q)$ on $\Lin(Q)$ induced by the conjugation action of $N_G(Q)$ on $Q$). Then $\LP=\bigcup_{i=1}^n\LP(Q,\delta_i)$ is clearly a disjoint union and $$\sum_{i=1}^n|N_G(Q):N_G(Q,\delta_i)|=|\Lin(Q)|=|Q:Q'|.$$
This suggests that perhaps Cossey's conjecture can be proved ``one vertex pair at a time",
and indeed, Cossey and Lewis established the following strong form of the conjecture
under certain conditions (see the proof of Theorem 1.2 of \cite{C2007} for $|G|$ odd,
and Theorem 3 of \cite{CL2010} for $Q$ abelian).

\begin{theorem}[Cossey-Lewis]\label{C-L}
Let $\varphi\in\IBr_p(G)$, where $G$ is $p$-solvable with $p>2$.
If either $|G|$ is odd or $Q$ is abelian, then $| \LP(Q,\delta)|\le |N_G(Q):N_G(Q,\delta)|$
for all $\delta\in\Lin(Q)$.
\end{theorem}

The purpose of the present paper is to investigate Cossey's conjecture in the case where $p=2$, for which there has been little progress.
Unfortunately, the above argument does not work in this case
because Lemma \ref{odd} can fail,
and the examples in \cite{C2006} and \cite{L2006} show that for a lift $\chi$ of some $\varphi\in\IBr_2(G)$ in a solvable group $G$,
two possibilities do happen: either $\chi$ has no linear vertex pair or all of the linear vertex pairs for $\chi$ are not conjugate.
This is a major obstacle to our research on lifts of $2$-Brauer characters.
In order to extend Lemma \ref{odd} to the case $p=2$,
we introduced the notion of {\bf twisted vertices},
and established the uniqueness of \emph{linear} twisted vertices under some conditions;
see \cite{WJ2022} for more details.
We can now use {\bf Navarro vertices} to deal with \emph{all} of the twisted vertices
for a lift of a given $2$-Brauer character in solvable groups,
and obtain an analogue of Lemma \ref{odd} for $p$ even.
For the definition of Navarro vertices, see \cite{N2002} or the next section.

The following is the first main result in this paper,
which strengthens Theorem A and Theorem B of \cite{WJ2022}.
Also, it is the starting point for our study of lifting Brauer characters for $p$ even,
just like Lemma \ref{odd} for $p$ odd.

\begin{thmA}
Let $G$ be a solvable group, and let $\chi\in\Irr(G)$ be a lift of some $\varphi\in\IBr_2(G)$.
Then the following hold.

{\rm (a)} $\chi$ has a linear Navarro vertex if and only if every vertex pair for $\chi$ is linear.
In that case, all of the twisted vertices for $\chi$ are linear and conjugate in $G$.

{\rm (b)} $\chi$ has a linear Navarro vertex if and only if every primitive character inducing $\chi$ has odd degree.

{\rm (c)} If $\chi$ is an $\mathcal{N}$-lift for some $2$-chain $\mathcal{N}$ of $G$,
then $\chi$ has a linear Navarro vertex.
\end{thmA}

Here, for a prime $p$ (not necessarily $2$), a {\bf $p$-chain} $\mathcal N$ of $G$ is a chain of normal subgroups of $G$:
$$\mathcal N=\{1=N_0\le N_1\le \cdots \le N_{k-1}\le N_k=G\}$$
such that each $N_i/N_{i-1}$ is either a $p$-group or a $p'$-group,
and $\chi\in\Irr(G)$ is an {\bf $\mathcal N$-lift} if for all $N\in\mathcal N$,
every irreducible constituent of $\chi_N$ is a lift of some irreducible $p$-Brauer character of $N$.
(See \cite{CL2011} for more details.)

As an application of Theorem A, we consider Cossey's conjecture for $p=2$.
Given $\varphi\in\IBr_2(G)$ with $G$ solvable,
we denote by $\LN$ the set of lifts of $\varphi$ that have a linear Navarro vertex,
which cannot be empty by Theorem A(a) of \cite{N2002},
and we write $\LN(Q,\delta)$ for the subset of
those characters in $\LN$ with twisted vertex $(Q,\delta)$.
By Theorem A, $\LN=\bigcup_{i=1}^n\LN(Q,\delta_i)$ is a disjoint union,
where $\{\delta_1,\ldots,\delta_n\}$ is a set of representatives of the $N_G(Q)$-orbits in $\Lin(Q)$ as before. Also, we can prove a weaker form of Cossey's conjecture for $p$ even ``one vertex pair at a time"; compare the following result with Theorem \ref{C-L}.

\begin{thmA}
Let $G$ be a solvable group, and let $\varphi\in\IBr_2(G)$ with vertex $Q$.
Then
$$|\LN(Q,\delta)|\le |N_G(Q):N_G(Q,\delta)|$$
for all $\delta\in\Lin(Q)$. In particular, $|\LN|\le |Q:Q'|$.
\end{thmA}

To prove Cossey's conjecture in the situation of Theorem B,
it is natural to ask when
it is true that $\LN$ and $\LP$ coincide (that is,
every lift of $\varphi$ has a linear Navarro vertex).
By Theorem A, $\LN=\LP$ if and only if
for any lift $\chi$ of $\varphi$, each primitive character inducing $\chi$ has odd degree.
As an easy consequence of Theorems A and B, we present the following
purely group-theoretic condition.

\begin{corA}
Let $G$ be a solvable group, and let $\varphi\in\IBr_2(G)$ with vertex $Q$.
Assume that there exist normal subgroups $N\le M$ of $G$ satisfying

{\rm (a)} $N$ has odd order,

{\rm (b)} all Sylow subgroups of $M/N$ are abelian,

{\rm (c)} $G/M$ is supersolvable.\\
Then $\LN=\LP$, and hence $|\LP|\le |Q:Q'|$.
\end{corA}

We end this introduction with some remarks.
As usual, we will work with Isaacs' $\pi$-partial characters in $\pi$-separable groups rather that Brauer characters in $p$-solvable groups.
For the reader's convenience, we will briefly review the definitions and basic properties of
$\pi$-partial characters and the associated vertices in the next section
(for more details see Isaacs' recent new book \cite{I2018}).
Here we emphasize that the ``classical'' case of Brauer characters is exactly the situation where
$\pi$ is the complement $p'$ of $\{p\}$ in the set of all prime numbers.
Actually, our proof of Theorem B is inspired by the arguments used by Cossey and Lewis
in the preprint \cite{CL2010} to handle the case $2\in\pi$.
Nevertheless, in our case where $2\notin\pi$ we need to use the $\pi$-induction of characters defined by Isaacs \cite{I1986} to replace ordinary induction,
and modify again the definition of Cossey's vertex pairs
for any $\chi\in\Irr(G)$ to obtain a new one which we call a $*$-vertex for $\chi$
(see Definition \ref{*}).
As expected, all of the $*$-vertices for every member of $\LN$ are still conjugate
(see Theorem \ref{unique}), which is also complementary to Lemma \ref{odd}
and is crucial for our goal,
and the proof is an immediate application of Theorem B of \cite{WJ2022}.

Throughout this paper, all groups are finite,
and most of the notation and results
can be found in \cite{I1976,N1998,I2018} for ordinary characters,
Brauer characters and Isaacs' $\pi$-partial characters.

\section{Preliminaries}
In this section, we briefly review some basic notions and results from Isaacs' $\pi$-theory
needed for our proofs.

\medskip
\noindent{\bf 2.1 \ $\pi$-Partial characters}

Following Isaacs \cite{I2018}, we fix a set $\pi$ of primes and let $G$ be a $\pi$-separable group. We write $G^0$ for the set of $\pi$-elements of $G$, and for any complex character $\chi$ of $G$, we use $\chi^0$ to denote the restriction of $\chi$ to $G^0$.
We call $\chi^0$ a {\bf $\pi$-partial character} of $G$.
If a $\pi$-partial character $\varphi$ cannot be written as a sum of two $\pi$-partial characters,
we say that $\varphi$ is {\bf irreducible} or that $\varphi$ is an {\bf $\I$-character}. The set of all $\I$-characters of $G$ is denoted $\I(G)$.
Also, if $\chi\in\Irr(G)$ satisfying $\chi^0=\varphi$ for some $\varphi\in\I(G)$,
then $\chi$ is called a {\bf $\pi$-lift}, or a {\bf lift} for short
when there is no risk of confusion.
We write
$$\LP=\{\chi\in\Irr(G)\,|\, \chi^0=\varphi\}$$
for the set of all lifts for $\varphi$.

By the Fong-Swan theorem,
it follows that if $\pi=p'$, the set of primes different from $p$,
then the $\pi$-partial characters of $G$ are exactly the $p$-Brauer characters in $p$-solvable groups. In this case, we have $\I(G)=\IBr_p(G)$.

Furthermore, Isaacs constructed a subset $\B(G)$ of $\Irr(G)$,
which is a canonical lift of $\I(G)$,
so the map $\chi\mapsto \chi^0$ defines a bijection from $\B(G)$ onto $\I(G)$
(see Theorem 5.1 of \cite{I2018}).

We now consider the vertices for $\I$-characters.
Given $\varphi\in\I(G)$, a {\bf vertex} for $\varphi$ is defined to be any Hall $\pi'$-subgroup of a subgroup $U$ of $G$ such that there exists $\theta\in\I(U)$, where $\theta^G=\varphi$ and $\theta(1)$ is a $\pi$-number.
We use $\I(G|Q)$ to denote the set of $\I$-characters of $G$ having $Q$ as a vertex.

The following result is fundamental.

\begin{lem}[Theorem 5.17 of \cite{I2018}]\label{vtx}
Let $G$ be $\pi$-separable, where $\pi$ is a set of primes, and let $\varphi\in\I(G)$.
Then all vertices for $\varphi$ form a single conjugacy class of $\pi'$-subgroups of $G$.
\end{lem}

\medskip
\noindent{\bf 2.2 \ $\pi$-Factored characters}

Let $\chi\in\Irr(G)$, where $G$ is a $\pi$-separable group.
We say that $\chi$ is {\bf $\pi$-special} if $\chi(1)$ is a $\pi$-number and
the determinantal order $o(\theta)$ is a $\pi$-number for every irreducible constituent $\theta$ of the restriction $\chi_S$ for every subnormal subgroup $S$ of $G$.
The set of $\pi$-special characters of $G$ is denoted $\X(G)$.

\begin{lem}[Theorem 2.10 of \cite{I2018}]\label{res}
Let $G$ be $\pi$-separable, and let $H\le G$ have $\pi'$-index.
Then restriction $\chi\mapsto\chi_H$ defines an injection from $\X(G)$ into $\X(H)$.
\end{lem}

\begin{lem}[Theorem 3.14 of \cite{I2018}]\label{lift}
Let $G$ be $\pi$-separable. Then the map $\chi\mapsto \chi^0$ defines an injection from $\X(G)$ into $\I(G)$.
\end{lem}

\begin{lem}[Theorem 2.2 of \cite{I2018}] \label{prod}
Let $G$ be $\pi$-separable,
and suppose that $\alpha,\beta\in\Irr(G)$ are $\pi$-special and $\pi'$-special, respectively.
Then $\alpha\beta$ is irreducible.
Also, if $\alpha\beta=\alpha'\beta'$, where $\alpha'$ is $\pi$-special and $\beta'$ is $\pi'$-special, then $\alpha=\alpha'$ and $\beta=\beta'$.
\end{lem}

If $\chi\in\Irr(G)$ can be written as $\chi=\alpha\beta$,
where $\alpha$ is $\pi$-special and $\beta$ is $\pi'$-special,
then $\chi$ is said to be {\bf $\pi$-factored}.
Since $\alpha$ and $\beta$ are uniquely determined by $\chi$,
we often use $\chi_\pi$ and $\chi_{\pi'}$ to replace $\alpha$ and $\beta$, respectively.

By definition, it is easy to see that every normal constituent of a $\pi$-special character is also $\pi$-special, and so we have the following.

\begin{lem}\label{pi-pi}
Let $G$ be $\pi$-separable, and let $N\NM G$.
If $\chi\in\Irr(G)$ is $\pi$-factored, then every irreducible constituent $\theta$ of $\chi_N$
is also $\pi$-factored. Furthermore, $\theta_\pi$ and $\theta_{\pi'}$ lie under $\chi_\pi$
and $\chi_{\pi'}$, respectively.
\end{lem}

Recall that if $\theta\in\Irr(H)$, where $H$ is a subgroup of $G$,
we write $\Irr(G|\theta)$ to denote the set of irreducible constituents of $\theta^G$,
that is, the set of irreducible characters of $G$ lying over $\theta$.

\begin{lem}\label{CL3.2}
Let $N\NM G$, where $G$ is $\pi$-separable, and suppose that $G/N$ is a $\pi'$-group.
Let $\alpha,\beta\in\Irr(N)$ be $\pi$-special and $\pi'$-special, respectively.
Then $\alpha$ is invariant in $G$ if and only if every member of $\Irr(G|\alpha\beta)$ is $\pi$-factored, and this happens precisely when one member of $\Irr(G|\alpha\beta)$ is $\pi$-factored.
\end{lem}
\begin{proof}
If $\alpha$ is invariant in $G$, then by Lemma 4.2 of \cite{I2018} every member of $\Irr(G|\alpha\beta)$ is $\pi$-factored.
Assume some character $\chi\in\Irr(G|\alpha\beta)$ is $\pi$-factored.
By Lemma \ref{pi-pi}, we see that $\chi_\pi$ lies over $\alpha$,
and since $G/N$ is a $\pi'$-group and $\chi_\pi$ has $\pi$-degree,
we have $(\chi_\pi)_N=\alpha$ by Clifford's theorem,
and thus $\alpha$ is invariant in $G$. The result now follows.
\end{proof}

\medskip
\noindent{\bf 2.3 \ Vertex pairs}

We continue to assume that $G$ is $\pi$-separable and $\chi\in\Irr(G)$.
Let $Q$ be a $\pi'$-subgroup of $G$ and $\delta\in\Irr(Q)$.
Following Cossey \cite{C2010}, the pair $(Q,\delta)$ is called a {\bf vertex pair} for $\chi$ if there exists a subgroup $U$ of $G$ such that $Q$ is a Hall $\pi'$-subgroup of $U$
and $\delta=(\psi_{\pi'})_Q$, where $\psi\in\Irr(U)$ is $\pi$-factored with $\psi^G=\chi$.
(Recall that we use $\psi_{\pi'}$ to denote the $\pi'$-special factor of $\psi$.)
We say that $(Q,\delta)$ is a {\bf linear vertex} if $\delta\in\Lin(Q)$ is a linear character of $Q$.
Clearly, all of the vertex pairs for $\chi$ need not be conjugate in $G$.

The importance of linear vertex pairs is illustrated by the following result.

\begin{lem}[Theorem 4.6(d) of \cite{C2010}]\label{vertex}
Let $G$ be a $\pi$-separable group, and let $\chi\in\Irr(G)$ be a lift of some $\varphi\in\I(G)$.
If $(Q,\delta)$ is a linear vertex pair for $\chi$, then $Q$ is a vertex for $\varphi$.
\end{lem}

\medskip
\noindent{\bf 2.4 \ Navarro vertices}

The definition of Navarro vertices relies on the following fundamental result,
whose proof can be seen Theorem 2.2 and Corollary 2.4 of \cite{N2002}.
Recall that if $\theta\in\Irr(N)$, where $N\NM G$ and $G$ is a group,
we often use $G_\theta$ to denote the inertia group of $\theta$ in $G$
and write $\chi_\theta$ for the Clifford correspondent of any $\chi\in\Irr(G|\theta)$
with respect to $\theta$.

\begin{lem}\label{N-nuc}
Let $G$ be a $\pi$-separable group, and let $\chi\in\Irr(G)$.
Then there is a unique normal subgroup $N$ of $G$ maximal with the property that
every irreducible constituent $\theta$ of $\chi_N$ is $\pi$-factored.
Furthermore, if $\theta$ is invariant in $G$, that is, if $G_\theta=G$, then $N=G$ and $\theta=\chi$.
\end{lem}

Now, the {\bf normal nucleus} $(W,\gamma)$ for $\chi\in\Irr(G)$ can be defined inductively.
If $\chi$ is $\pi$-factored, then we let $(W,\gamma)=(G,\chi)$.
Otherwise, let $N$ and $\theta$ be as in Lemma \ref{N-nuc}.
Then $G_\theta<G$, and we define $(W,\gamma)$ to be the normal nucleus for $\chi_\theta$. It is easy to see that $\gamma$ is $\pi$-factored with $\gamma^G=\chi$,
and that the normal nucleus for $\chi$ is uniquely determined up to conjugacy.

Furthermore, let $Q$ be a Hall $\pi'$-subgroup of $W$ and let $\delta=(\gamma_{\pi'})_Q$.
Then the pair $(Q,\delta)$ is clearly a vertex pair for $\chi$,
which is called a {\bf Navarro vertex} for $\chi$.
By the construction of normal nuclei, it is easy to see that
all of the Navarro vertices for $\chi$ are conjugate in $G$.

\medskip
\noindent{\bf 2.5 \ $\pi$-Induction}

Assume that $G$ is a $\pi$-separable group with $2\notin\pi$. For each subgroup $H$ of $G$,
Isaacs defined the $\pi$-standard sign character $\delta_{(G,H)}$,
which is a linear character of $H$  that has values $\pm 1$.
For definition and properties of this sign character,
the reader is referred to \cite{I1986}.

\begin{lem}[Theorem 2.5 of \cite{I1986}]\label{1}
Let $G$ be a $\pi$-separable group, where $2\notin\pi$.

{\rm (1)} If $H\le K\le G$, then $\delta_{(G,H)}=(\delta_{(G,K)})_H\delta_{(K,H)}$.

{\rm (2)} If $N\le H\le G$ and $N\NM G$,
then $N$ is contained in the kernel of $\delta_{(G,H)}$.
\end{lem}

Using the $\pi$-standard sign character, Isaacs \cite{I1986} introduced the notion of $\pi$-induction.

\begin{defn}
Let $G$ be $\pi$-separable with $2\notin\pi$,
and suppose that $\theta$ is a character of some subgroup $H$ of $G$.
Write $\theta^{\pi G}=(\delta_{(G,H)}\theta)^G$.
We say that $\theta^{\pi G}$ is the {\bf $\pi$-induction} of $\theta$ to $G$.
\end{defn}

\begin{lem}[Lemma 7.4 of \cite{I1986}]\label{3}
Let $G$ be $\pi$-separable with $2\notin\pi$, and let $\chi\in\Irr(G)$.
Suppose that $H\leq K\leq G$ and $\theta$ is a character of $H$.
Then $(\theta^{\pi K})^{\pi G}=\theta^{\pi G}$.
\end{lem}

The next result is the Clifford correspondence for $\pi$-induction.

\begin{lem}[Lemma 7.5 of \cite{I1986}]\label{4}
Let $G$ be $\pi$-separable with $2\notin\pi$,
and let $N\NM G$ and $\theta\in\Irr(N)$.
Then the map $\psi\mapsto\psi^{\pi G}$ defines a bijection
$\Irr(G_\theta|\theta)\to\Irr(G|\theta)$.
\end{lem}

The following is fundamental when studying induction of $\pi$-special characters.

\begin{lem}[Theorem 2.29 of \cite{I2018}]\label{spec-ind}
Let $\psi\in\Irr(U)$, where $U\le G$ and $G$ is $\pi$-separable,
and assume that every irreducible constituent of $\psi^G$ is $\pi$-special.
Then $|G:U|$ is a $\pi$-number and $\psi(1)$ is a $\pi$-number.
Also, $\psi$ is $\pi$-special if $2\in\pi$, and $\delta_{(G,U)}\psi$ is $\pi$-special if $2\notin\pi$.
\end{lem}

\section{Proof of Theorem A}
We begin with a lemma, which is of fundamental importance for the proof of our Theorem A.

\begin{lem}\label{lin-ver}
Let $\chi\in\Irr(G)$ be a lift of $\varphi\in\I(G)$,
where $G$ is $\pi$-separable with $2\notin\pi$.
Then the following are equivalent.

{\rm (a)} All of the vertex pairs for $\chi$ are linear.

{\rm (b)} Every quasi-primitive character $\psi$ inducing $\chi$ has odd degree,
where both $\psi_\pi$ and $\psi_{\pi'}$ are primitive.

{\rm (c)} Every $\pi$-factored character $\psi\in\Irr(U)$ inducing $\chi$
has odd degree, where $U$ is a subgroup of $G$ containing $O_{\pi}(G)O_{\pi'}(G)$ and $\psi_{\pi'}$ is primitive.
\end{lem}

\begin{proof}
Note that by definition, (a) is equivalent to saying that
each $\pi$-factored character that induces $\chi$ has $\pi$-degree.
To complete the proof, we fix a subgroup $U$ of $G$, and let $\psi\in\Irr(U)$ be $\pi$-factored with $\psi^G=\chi$. For notational simplicity, we write $\alpha=\psi_\pi$ and $\beta=\psi_{\pi'}$, so that $\psi=\alpha\beta$.

We claim that $\psi(1)$ is a $\pi$-number if and only if it is odd.
To see this, note that $2\notin\pi$, so $\psi(1)$ is odd whenever it is a $\pi$-number.
Conversely, if $\psi(1)$ is odd, then $\beta(1)$ is also odd,
and since
$$(\alpha^0\beta^0)^G=(\psi^0)^G=(\psi^G)^0=\chi^0=\varphi,$$
it follows that $\beta^0\in\I(U)$.
Now Corollary 2.14 of \cite{I2018} implies that $\beta^0$ is rational valued,
and Lemma 5.4 of \cite{I2018} tells us that $\beta^0$ must be principal.
Thus $\beta$ is linear, and $\psi(1)=\alpha(1)$ is a $\pi$-number, as claimed.
In particular, this proves that (a) implies both (b) and (c).

Now assume (b). To establish (a), it suffices to show that $\beta$ is linear.
If $\alpha=\sigma^U$, where $\sigma\in\Irr(X)$ is primitive and $X\le U$,
then $\chi=(\alpha\beta)^G=((\sigma\beta_X)^U)^G=(\sigma\beta_X)^G$.
Note that $|U:X|$ divides $\alpha(1)$, which is a $\pi$-number,
so $\beta_X$ is $\pi'$-special.
Furthermore, by Lemma \ref{spec-ind}, we see that $\delta_{(U,X)}\sigma$ is $\pi$-special, so $\sigma\beta_X=(\delta_{(U,X)}\sigma)(\delta_{(U,X)}\beta_X)$ is $\pi$-factored. To prove that $\beta$ is linear, we can assume, therefore, that $\alpha$ is primitive with $\sigma\beta_X$ in place of $\psi$ and $X$ in place of $U$.
Similarly, if $\beta=\rho^U$, where $\rho\in\Irr(Y)$ is primitive and $Y\le U$.
Then $\alpha_Y$ is $\pi$-special as $|U:Y|$ is a $\pi'$-number,
and $\rho$ is $\pi'$-special by Lemma \ref{spec-ind} again.
Now $\chi=(\alpha\rho^U)^G=(\alpha_Y\rho)^G$
and $\alpha_Y\rho$ is $\pi$-factored.
If we can prove that $\rho$ is linear, then $\rho^0$ must be principal,
and since $(\rho^0)^U=(\rho^U)^0=\beta^0\in\I(U)$ (see the previous paragraph),
it follows that $U=Y$, and thus $\beta=\rho$ is linear.
We may assume, therefore, that $\beta$ is primitive
with $\alpha_Y\rho$ in place of $\psi$ and $Y$ in place of $U$.
Now we repeat this progress until both $\alpha$ and $\beta$ are primitive.
It is clear that primitive characters in any group are always quasi-primitive,
and that a $\pi$-factored character in a $\pi$-separable group is quasi-primitive if and only if
its $\pi$-special factor and $\pi'$-special factor are all quasi-primitive (by basic properties of $\pi$-special characters).
Thus $\psi=\alpha\beta$ is quasi-primitive, and then (a) follows by (b).

Finally, assume (c). Write $N=O_\pi(G)O_{\pi'}(G)$.
Then each irreducible character of $N$ is clearly $\pi$-factored. We claim that $|NU:U|$ is a $\pi$-number. To see this, let $\hat\varphi\in\B(G)$ be a lift of $\varphi$,
so that every irreducible constituent $\theta$ of $\hat\varphi_N$ is also a $\B$-character.
Note that $\theta\in\Irr(N)$ is $\pi$-factored, so $\theta$ must be $\pi$-special (see Theorem 4.12 of \cite{I2018}), and thus $\theta^0\in\I(N)$ has $\pi$-degree
lying under $\varphi$. Moreover, since $(\psi^0)^G=(\psi^G)^0=\chi^0=\varphi$, it follows by Lemma 5.21 of \cite{I2018} that $|NU:U|$ is a $\pi$-number, as claimed.
Writing $\xi=\psi^{NU}$, we see by Lemma 3.1 of \cite{WJ2022}
that $\xi$ is $\pi$-factored with $(\xi_{\pi'})_U=\delta_{(NU,U)}\beta$ .
To establish (a), we may assume further that $\psi$ is quasi-primitive and $\beta$ is primitive
by the equivalence of (a) with (b) just proved.
Then $\xi_{\pi'}$ is also primitive with degree $\beta(1)$, so (a) follows by (c)
with $\xi$ in place of $\psi$ and $NU$ in place of $U$.
\end{proof}

Following the terminology of Cossey and Lewis \cite{CL2011},
a {\bf $\pi$-chain} $\mathcal N$ of $G$ is a chain of normal subgroups of $G$:
$$\mathcal N=\{1=N_0\le N_1\le \cdots \le N_{k-1}\le N_k=G\}$$
such that $N_i/N_{i-1}$ is either a $\pi$-group or a $\pi'$-group for $i=1,\dots,k$, and $\chi\in\Irr(G)$ is an {\bf $\mathcal N$-lift} if every irreducible constituent of $\chi_N$ is a lift for all $N\in\mathcal N$.

\begin{theorem}\label{ver-lin}
Let $G$ be $\pi$-separable with $2\notin\pi$, and let $\chi\in\Irr(G)$ be a lift
of $\varphi\in\I(G)$. Assume one of the following conditions.

{\rm (a)} $\chi$ has a linear Navarro vertex.

{\rm (b)} $\chi$ is an $\mathcal{N}$-lift for some $\pi$-chain $\mathcal{N}$ of $G$.\\
Then all of the vertex pairs for $\chi$ are linear.
\end{theorem}

\begin{proof}
By Lemma \ref{lin-ver}, it suffices to show that
every $\pi$-factored character $\psi\in\Irr(U)$ that induces $\chi$ has odd degree,
where $U\le G$ and $\psi_{\pi'}$ is primitive.
We write $\alpha=\psi_\pi$ and $\beta=\psi_{\pi'}$, so that $\psi=\alpha\beta$.

Suppose first that $\chi$ is $\pi$-factored.
To show the theorem in this situation, it suffices to establish that $\chi(1)$ is a $\pi$-number
as $2\notin\pi$.
In case (b), this is true by Lemma 2.3(2) of \cite{WJ2022},
and in case (a), since $(G,\chi)$ is itself a normal nucleus for $\chi$,
it follows by the definition of Navarro vertices that $\chi$ has $\pi$-degree.

Now suppose that $\chi$ is not $\pi$-factored.
Reasoning as we did in the proof of Theorem 3.4 of \cite{WJ2022},
there exists a normal subgroup $N$ of $G$
such that every irreducible constituent $\theta$ of $\chi_N$ is $\pi$-factored with $\pi$-degree
and $G_\theta<G$.
(Specifically, in case (a) we let $N$ be the unique maximal normal subgroup of $G$ such that
$\chi_N$ has $\pi$-factored irreducible constituents,
and in case (b) we choose $N\in\mathcal N$ maximal with the property that
all of the irreducible constituents of $\chi_N$ are $\pi$-factored.)
Furthermore, the Clifford correspondent $\chi_\theta$ of $\chi$ over $\theta$ also
satisfies the both conditions on $\chi$ (by replacing $G$ with $G_\theta$).
Observe that $\theta^0$ is irreducible with $\pi$-degree, and that $\theta^0$ lies under $\varphi$. Since $(\psi^0)^G=\chi^0=\varphi$,
we deduce by Lemma 5.21 of \cite{I2018} that $|NU:U|$ is a $\pi$-number.
Now Lemma 3.1 of \cite{WJ2022} applies, and writing $\xi=\psi^{NU}$,
we know that $\xi$ is $\pi$-factored with $(\xi_{\pi'})_U=\delta_{(NU,U)}\beta$.
In particular, $\beta$ is linear if and only if $\xi_{\pi'}$ is.
It is easy to see that $\xi_{\pi'}$ is necessarily primitive as it is an extension of
the primitive character $\delta_{(NU,U)}\beta$ (by Mackey).
Now replacing $U$ by $NU$ and $\psi$ by $\xi$,
we may assume further that $N\le U$.
Moreover, we can replace $\theta$ by a conjugate and assume that $\theta$ lies under $\psi$.

We consider the Clifford correspondents $\psi_\theta\in\Irr(U_\theta)$ and $\chi_\theta\in\Irr(G_\theta)$ of $\psi$ and $\chi$ over $\theta$, respectively.
Note that $N\NM U$ and that both $\psi$ and $\theta$ are $\pi$-factored,
so $\theta_\pi$ lies under $\alpha$ and $\theta_{\pi'}$ lies under $\beta$.
Since $\beta$ is assumed to be primitive, it follows that $\beta_N$ is a multiple of $\theta_{\pi'}$, and hence $\theta_{\pi'}$ is invariant in $U$.
From this, we deduce that $U_\theta$ is exactly the inertia group of $\theta_\pi$ in $U$.
Let $\tilde\alpha$ be the Clifford correspondent of $\alpha$ over $\theta_\pi$,
so that $\tilde\alpha^U=\alpha$.
By Lemma \ref{spec-ind}, we see that $|U:U_\theta|$ is a $\pi$-number and $\delta_{(U,U_\theta)}\tilde\alpha$ is $\pi$-special.
Furthermore, we have that $\tilde\beta\in\Irr(U_\theta)$ is $\pi'$-special, where $\tilde\beta$ denotes the restriction of $\beta$ to $U_\theta$.
Now $(\tilde\alpha\tilde\beta)^U=\tilde\alpha^U\beta=\alpha\beta=\psi$,
so $\tilde\alpha\tilde\beta$ is irreducible, which clearly lies over $\theta_\pi\theta_{\pi'}=\theta$.
It follows that $\tilde\alpha\tilde\beta$ is the Clifford correspondent of $\psi$ over $\theta$,
that is, $\psi_\theta=\tilde\alpha\tilde\beta$.
Also, we have $\psi_\theta=(\delta_{(U,U_\theta)}\tilde\alpha)(\delta_{(U,U_\theta)}\tilde\beta)$,
and thus $\psi_\theta$ is $\pi$-factored.
Note that $(\psi_\theta)^{G_\theta}=\chi_\theta$ and $G_\theta<G$.
By induction on $|G|$, we conclude that $\psi_\theta$ has $\pi$-degree,
and thus $\psi(1)=|U:U_\theta|\psi_\theta(1)$ is a $\pi$-number.
The proof is now complete.
\end{proof}

We can now prove Theorem A in the introduction,
which is a special case of the following more general result (by taking $\pi$ to be the set of all odd primes).
Recall that quasi-primitive characters of solvable groups are primitive (see Theorem 11.33 of \cite{I1976}).

\begin{theorem}\label{A}
Let $G$ be a $\pi$-separable group with $2\notin\pi$, and let $\chi\in\Irr(G)$ be a lift of some $\varphi\in\I(G)$. Then the following hold.

{\rm (a)} $\chi$ has a linear Navarro vertex if and only if every vertex pair for $\chi$ is linear.
In that case, all of the twisted vertices for $\chi$ are linear and conjugate in $G$.

{\rm (b)} $\chi$ has a linear Navarro vertex if and only if every quasi-primitive character inducing $\chi$ has odd degree.

{\rm (c)} If $\chi$ is an $\mathcal{N}$-lift for some $\pi$-chain $\mathcal{N}$ of $G$,
then $\chi$ has a linear Navarro vertex.
\end{theorem}
\begin{proof}
The result follows by Lemma \ref{lin-ver} and Theorem \ref{ver-lin}, together with Theorem B of \cite{WJ2022}.
\end{proof}

\section{A new modification of vertex pairs}
We first need to extend the definition of the set $\LN$, which appeared in the introduction,
and for simplicity, we introduce a notion.

\begin{defn}\label{good}
Let $G$ be $\pi$-separable, and let $\chi\in\Irr(G)$.
We say that $\chi$ is a {\bf good lift} if $\chi$ is a lift for some $\varphi\in\I(G)$
with a linear Navarro vertex.
The set of all good lifts for $\varphi$ is denoted $\LN$.
\end{defn}

Of course, all of the lifts of a $\pi$-separable group are good whenever $2\in\pi$ by Lemma 4.3 of \cite{CL2012}, so this notion is useful only for $2\notin\pi$.

The following is a fundamental result.

\begin{lem}\label{good-trans}
Let $\chi\in\Irr(G)$ be a good lift, where $G$ is $\pi$-separable with $2\notin\pi$.
If $\chi=\psi^G$ with $\psi\in\Irr(U)$ and $U\le G$,
then $\psi$ is also a good lift.
\end{lem}
\begin{proof}
It is clear that each vertex pair of $\psi$ is also a vertex pair of $\chi$,
and thus the result follows by Theorem \ref{A}.
\end{proof}

\begin{lem}\label{11}
Let $\chi\in\Irr(G)$ be a good lift, where $G$ is $\pi$-separable with $2\notin\pi$,
and let $N\NM G$.
If $\theta\in\Irr(N)$ is a $\pi$-factored constituent of $\chi_N$,
then $\theta(1)$ is a $\pi$-number. In particular, $\theta$ is a lift.
\end{lem}
\begin{proof}
By Lemma \ref{N-nuc}, there exists a normal nucleus $(W,\gamma)$ for $\chi$ such that
$N\le W$ and $\theta$ lies under $\gamma$.
Then $\gamma\in\Irr(W)$ is $\pi$-factored with $\gamma^G=\chi$.
Let $Q$ be a Hall $\pi'$-subgroup of $W$ and write $\delta=(\gamma_{\pi'})_Q$.
By definition, we see that $(Q,\delta)$ is a vertex pair for $\chi$.
Since $\chi$ is a good lift, we have $\delta(1)=1$,
and thus $\gamma(1)=\gamma_\pi(1)$, which is a $\pi$-number.
Note that $\theta(1)$ divides $\gamma(1)$,
so $\theta(1)$ is a $\pi$-number.
In particular, $\theta_{\pi'}$ is linear, which implies that
$\theta^0=(\theta_\pi)^0\in\I(N)$ by Lemma \ref{lift}.
The proof is complete.
\end{proof}

We now modify the notion of vertex pairs, as mentioned in the Introduction.

 \begin{defn}\label{*}
Let $\chi\in\Irr(G)$, where $G$ is a $\pi$-separable group with $2\notin\pi$,
and suppose that $Q$ is a $\pi'$-subgroup of $G$ and $\delta\in\Irr(Q)$.
We say that $(Q,\delta)$ is a {\bf $*$-vertex} for $\chi$ if there exist a subgroup $U\le G$ and a $\pi$-factored character $\psi\in\Irr(U)$ with $\chi=\psi^{\pi G}$,
and such that $Q$ is a Hall $\pi'$-subgroup of $U$ and $\delta=(\psi_{\pi'})_Q$.
\end{defn}

Our motivation for this is the following result (compare with Lemma \ref{odd}).
For the notion of a twisted vertex of a character $\chi\in\Irr(G)$,
we refer the reader to Definition 2.1 and the remark after Lemma 2.2 of \cite{WJ2022}.

\begin{theorem}\label{unique}
Let $G$ be a $\pi$-separable group with $2\notin\pi$,
$Q$ a $\pi'$-subgroup of $G$ and $\delta\in\Lin(Q)$.
If $\chi\in\Irr(G)$ is a good lift, then $\chi$ has twisted vertex $(Q,\delta)$
if and only if $\chi$ has $*$-vertex $(Q,\delta_{(G,Q)}\delta)$.
In particular, all $*$-vertices for $\chi$ are linear and conjugate in $G$.
\end{theorem}
\begin{proof}
Let $(Q,\delta)$ be any twisted vertex for $\chi$.
Then $\delta$ is linear by Theorem \ref{A}.
By definition, there exist $U\le G$ and $\psi\in\Irr(U)$ with $\chi=\psi^G$,
and such that $Q$ is a Hall $\pi'$-subgroup of $U$, $\psi$ is $\pi$-factored and $\delta=\delta_{(U,Q)}(\psi_{\pi'})_Q$.
On the other hand, notice that $\chi=\psi^G=(\delta_{(G,U)}\psi)^{\pi G}$ and
that $\delta_{(G,U)}\psi_{\pi'}$ is $\pi'$-special. It follows that $(Q,\delta^*)$ is a $*$-vertex for $\chi$, where $\delta^*=(\delta_{(G,U)}\psi_{\pi'})_Q$.
By Lemma \ref{1}, $\delta_{(U,Q)}=\delta_{(G,Q)}(\delta_{(G,U)})_Q$,
so $\delta^*=\delta_{(G,Q)}\delta$.
Therefore, $(Q,\delta)\mapsto (Q,\delta_{(G,Q)}\delta)$
is a well-defined map between the set of twisted vertices for $\chi$ and the set of $*$-vertices for $\chi$. Moreover, it is easy to see that this map is bijective.

Furthermore, since $\chi$ has a linear Navarro vertex,
we conclude from Theorem \ref{A} again that all of the twisted vertices for $\chi$ are linear and conjugate, and hence all of the $*$-vertices of $\chi$ are conjugate in $G$.
This completes the proof.
\end{proof}

\section{Proofs of Theorem B and Corollary C}
We begin by introducing some more notation.
Let $G$ be a $\pi$-separable group with $2\notin\pi$.
Suppose that $Q$ is a $\pi'$-subgroup of $G$ and $\varphi\in\I(G|Q)$,
that is, $\varphi\in\I(G)$ has vertex $Q$.
By Lemma \ref{vertex} and Theorem \ref{unique},
every good lift $\chi$ for $\varphi$ has a $*$-vertex of the form $(Q,\delta)$
for some $\delta\in\Lin(Q)$.
For notational convenience, we will write $\LPS(Q,\delta)$ to denote the set of those
good lifts for $\varphi$ that have $(Q,\delta)$ as a $*$-vertex, so that
$\LN=\bigcup_{\delta} \LPS(Q,\delta)$,
where $\delta$ runs over a set of representatives of the $N_G(Q)$-orbits in $\Lin(Q)$.

We need to investigate the behavior of characters in $\LPS(Q,\delta)$ with respect to normal subgroups (compare with Lemma 3.1 of \cite{CL2010}).

\begin{lem}\label{CL3.1}
Let $\varphi\in\I(G)$, where $G$ is a $\pi$-separable group with $2\notin\pi$,
and let $Q$ be a $\pi'$-subgroup of $G$ with $\delta\in\Lin(Q)$.
Assume that $\chi\in \LPS(Q,\delta)$ and $N\NM G$.
If every irreducible constituent of $\chi_N$ is $\pi$-factored,
then the following hold.

{\rm (1)} $Q\cap N$ is a Hall $\pi'$-subgroup of $N$.

{\rm (2)} There exists a unique $\pi'$-special character $\beta$ of $N$
such that $\beta_{Q\cap N}=\delta_{Q\cap N}$.
In particular, $\beta$ is linear.

{\rm (3)} There is a $\pi$-special character $\alpha\in\Irr(N)$
such that $\alpha\beta$ is an irreducible constituent of $\chi_N$,
where $\beta$ is as in (2).
Furthermore $\alpha^0=(\alpha\beta)^0$ is an irreducible constituent of $\varphi_N$.

{\rm (4)} Every irreducible constituent of $\chi_N$ is a lift
for some irreducible $\pi$-partial character of $N$ with $\pi$-degree,
and every irreducible constituent of $\varphi_N$ can be lifted to an irreducible constituent of $\chi_N$.
\end{lem}
\begin{proof}
By Theorem \ref{unique}, all $*$-vertices for $\chi$ form a single $G$-conjugacy class,
and so there exists a normal nucleus $(W,\gamma)$ of $\chi$
where $\gamma\in\Irr(W)$ is $\pi$-factored with $\chi=\gamma^G=(\delta_{(G,W)}\gamma)^{\pi G}$,
and such that $Q$ is a $\Hall$ $\pi'$-subgroup of $W$ and $\delta=(\delta_{(G,W)}\gamma_{\pi'})_Q$.
Applying Lemma \ref{N-nuc}, we have $N\le W$, and thus $Q\cap N$ is a $\Hall$ $\pi'$-subgroup of $N$. This establishes (1).

For (2), let  $\beta=(\gamma_{\pi'})_N$, and note that $\gamma_{\pi'}(1)=\delta(1)=1$,
so $\beta$ is a linear $\pi'$-special character of $N$.
By Lemma \ref{1}(2), we see that $N\leq\Ker(\delta_{(G,W)})$, and hence
$$\beta_{Q\cap N}=((\delta_{(G,W)}\gamma_{\pi'})_N)_{Q\cap N}=((\delta_{(G,W)}\gamma_{\pi'})_Q)_{Q\cap N}=\delta_{Q\cap N}.$$
Now the uniqueness of $\beta$ is guaranteed by Lemma \ref{res}
with the roles of $\pi$ and $\pi'$ interchanged.

To show (3), let $\alpha$ be an irreducible constituent of $(\gamma_\pi)_N$.
Then $\alpha$ is $\pi$-special, so $\alpha\beta$ is irreducible by Lemma \ref{prod}.
It is clear that $\alpha\beta$ is an irreducible constituent of $\gamma_N$,
and since $\gamma$ lies under $\chi$, it follows that $\alpha\beta$ is also an irreducible constituent of $\chi_N$.
Note that $\beta\in\Lin(N)$ is $\pi'$-special, which implies that $\beta^0\in\I(N)$ is principal,
and thus $\alpha^0=(\alpha\beta)^0$ is a constituent of
$(\chi_N)^0=(\chi^0)_N=\varphi_N$.

Finally, we establish (4). By (3), we see that $\alpha\beta$ is an irreducible constituent of $\chi_N$ and that it is a lift.
Since every irreducible constituent of $\chi_N$ is conjugate to $\alpha\beta$
by Clifford' theorem, the first statement follows.
This, along the formula $\varphi_N=(\chi^0)_N=(\chi_N)^0$,
implies that the second statement holds.
\end{proof}

As an application, we establish a result that is similar to Lemma 3.3 of \cite{CL2010} for $*$-vertices. Although the proof is similar, we present it here for completeness.

\begin{lem}\label{CL3.3}
Let $G$ be $\pi$-separable with $2\notin\pi$,
and suppose that $\chi,\psi\in \LPS(Q,\delta)$,
where $\varphi\in\I(G)$ and $Q$ is a $\pi'$-subgroup of $G$ with $\delta\in\Lin(Q)$.
If $N$ is a normal subgroup of $G$,
then the irreducible constituents of $\chi_N$ are $\pi$-factored if and only if the irreducible constituents of $\psi_N$ are $\pi$-factored.
\end{lem}
\begin{proof}
By symmetry, it is no loss to assume that each irreducible constituent of $\chi_N$ is $\pi$-factored. We need to show that the irreducible constituents of $\psi_N$ are also $\pi$-factored.

Let $L\NM G$ be maximal with the property that $L\le N$ and the irreducible constituents of $\psi_L$ are $\pi$-factored. If $L=N$, we are done, and so we suppose that $L<N$.

Note that the irreducible constituents of $\chi_L$ are also $\pi$-factored,
and by Lemma \ref{CL3.1}, there exist $\alpha_1,\alpha_2\in\X(L)$ and $\beta\in\Y(L)$
such that $\alpha_1\beta$ and $\alpha_2\beta$ are irreducible constituents of $\chi_L$ and $\psi_L$, respectively, that $\beta_{Q\cap L}=\delta_{Q\cap L}$, and that both $(\alpha_1)^0$ and $(\alpha_2)^0$ are irreducible constituents of $\varphi_L$.
Then $(\alpha_1)^0$ and $(\alpha_2)^0$ are conjugate in $G$ by Clifford's theorem for $\pi$-partial characters (see Corollary 5.7 of \cite{I2018}), which implies that $\alpha_1$ and $\alpha_2$ are $G$-conjugate
by the injection from $\X(L)\to\I(L)$ (see Lemma \ref{lift}).

Let $M/L$ be a chief factor of $G$ with $M\le N$.
Then $M/L$ is either a $\pi$-group or a $\pi'$-group,
and since we are assuming that $\chi_N$ has $\pi$-factored irreducible constituents,
each irreducible constituent of $\chi_M$ is $\pi$-factored,
and hence some member of $\Irr(M|\alpha_1\beta)$ must be $\pi$-factored.
By Lemma \ref{CL3.2}, we see that either $\beta$ or $\alpha_1$ is $M$-invariant, according as $M/L$ is a $\pi$-group or a $\pi'$-group.
Clearly $\alpha_1$ is $M$-invariant if and only if $\alpha_2$ (which is $G$-conjugate to $\alpha_1$) is $M$-invariant.
In both cases, we conclude that every member of $\Irr(M|\alpha_2\beta)$  is $\pi$-factored
by Lemma \ref{CL3.2} again. Since  $\Irr(M|\alpha_2\beta)$ contains some
irreducible constituent of $\psi_M$,
we know that the irreducible constituents of $\psi_M$ are also $\pi$-factored. This contradicts the maximality of $L$, thus proving the lemma.
\end{proof}

The following is part of Theorem 2.1 of \cite{L2010}
in which the whole group $G$ was originally supposed to be solvable.
Actually, it suffices to assume that $G$ is $\pi$-separable for our purpose,
and we omit the proof here because there is no need to repeat the argument word by word.

\begin{lem}\label{L2010-2.1}
Let $G$ be a $\pi$-separable group, and suppose that $\varphi\in\I(G)$ has vertex $Q$,
where $Q$ is a $\pi'$-subgroup of $G$.
Let $V$ be a subgroup of $G$ containing $Q$, and write $\VI(V|Q)$ to denote the set of irreducible $\pi$-partial characters of $V$ with vertex $Q$ that induce $\varphi$, that is,
$$\VI(V|Q)=\{\eta\in\I(V|Q) \,|\, \eta^G=\varphi\}.$$
If $G$ and $V$ are chosen so that $|G|+|V|$ is minimal subject to the condition that
$$|\VI(V|Q)|>|N_G(Q):N_V(Q)|,$$
then $V$ is a nonnormal maximal subgroup of $G$,
and $\varphi_N$ has a unique irreducible constituent,
where $N$ is the core of $V$ in $G$.
\end{lem}

A key ingredient in our proof of Theorem B is the following corollary of
Lemma \ref{L2010-2.1}.

\begin{cor}\label{L2010-2.2}
Let $G$ be $\pi$-separable with $2\notin\pi$, and let $Q$ be a $\pi'$-subgroup of $G$.
Suppose that $\varphi\in\I(G|Q)$ and $Q\leq V\leq G$.
Then $|\VI(V|Q)|\leq |N_G(Q):N_V(Q)|$.
\end{cor}
\begin{proof}
Assume that the result is false, and let $G$ be a counterexample such that
$|G|+|V|$ is as small as possible. Let $N$ be the core of $V$ in $G$.
Applying Lemma \ref{L2010-2.1}, we see that $V$ is a maximal subgroup of $G$
with $N<V$, and that $\varphi_N=e\alpha$ where $e$ is a positive integer and $\alpha\in\I(N)$.
In particular, $\alpha$ is invariant in $G$.

Let $K/N$ be a chief factor of $G$.
Then $K\nsubseteq V$ by the definition of $N$,
and the maximality of $V$ forces $KV=G$. Write $D=K\cap V$.
Since $G$ is $\pi$-separable, we know that $K/N$ is either a $\pi$-group or a $\pi'$-group.

Suppose that $K/N$ is a $\pi'$-group.
Let $\theta\in\I(K)$ lie under $\varphi$ and over $\alpha$.
Since $\alpha$ is $G$-invariant, we have $\theta_N=\alpha$ by Clifford's theorem for $\I$-characters (Corollary 5.7 of \cite{I2018}), and thus $\theta_D$ is irreducible.
Then restriction defines a bijection from $\I(G|\theta)$ onto $\I(V|\theta_D)$ (see Lemma 5.20 of \cite{I2018}), which implies that $\varphi_V$ is irreducible,
and consequently, $\varphi$ cannot be induced from $V$.
This shows that the set $\VI(V|Q)$ is empty, a contradiction.
So we are done in this case.

Now suppose that $K/N$ is a $\pi$-group.
Then $K/N$ has odd order because $2\notin\pi$,
and hence $K/N$ is solvable by the Feit-Thompson odd-order theorem.
We conclude that $K/N$ is an abelian $p$-group for some odd prime $p$,
which implies that $D=K\cap V$ is normal in $G$, so $D=N$.

We will use $\B$-characters to derive a contradiction
(see \cite{I2018} for the definition and properties).
Let $\hat{\alpha}\in \B(N)$ be such that $(\hat{\alpha})^0=\alpha$.
Since $\alpha$ and $\hat\alpha$ determine each other,
it follows that $\hat{\alpha}$ is also $G$-invariant.
Fix $\xi\in \VI(V|Q)$. Then $\xi^G=\varphi$,
which implies that $\xi$ lies under $\varphi$ (see Lemma 5.8 of \cite{I2018}), and thus $\xi$ lies over $\alpha$.
Let $\hat\xi\in\B(V)$ be a lift for $\xi$.
Since the irreducible constituents of $\hat\xi_N$ are still $\B$-characters,
we see that $\hat\alpha$ is an irreducible constituent of ${\hat\xi}_N$.
Furthermore, observe that $(\hat\xi^G)^0=(\hat\xi^0)^G=\xi^G=\varphi$,
so $\hat\xi^G$ is irreducible, which is not the case by Theorem 7.21 of \cite{I2018}.
This contradiction completes the proof.
\end{proof}

The following result, which we think is of independent interest,
is the key to the proof of Theorem B in the introduction.

\begin{theorem}\label{B'}
Let $G$ be a $\pi$-separable group with $2\notin\pi$,
and suppose that $\varphi\in\I(G)$ has vertex $Q$, where $Q$ is a $\pi'$-subgroup of $G$.
If $\delta\in\Lin(Q)$, then $|\LPS(Q,\delta)|\leq|N_G(Q):N_G(Q,\delta)|$.
\end{theorem}
\begin{proof}
We proceed by induction on $|G|$, and assume without loss that $\LPS(Q,\delta)$ is not an empty set.
By Lemma \ref{N-nuc} and Lemma \ref{CL3.3},
we can fix $N\NM G$ such that for each $\chi\in \LPS(Q,\delta)$,
$N$ is the unique maximal normal subgroup of $G$ with the property that
every irreducible constituent of $\chi_N$ is $\pi$-factored.
We will complete the proof by carrying out the following steps.

\HJ{Step 1}
We may suppose that $\varphi_N=e\mu$, where $e$ is an integer and $\mu\in\I(N)$.
\EHJ
\begin{proof}
By Lemma 6.33 of \cite{I2018}, we can choose an irreducible constituent $\mu$ of $\varphi_N$ such that the Clifford correspondent $\varphi_\mu$ of $\varphi$ with respect to $\mu$ also has vertex $Q$, that is, $\varphi_\mu\in\I(G_\mu|Q)$.
Set $T=G_\mu$.
Consider the conjugation action of $N_G(Q)$ on $\Lin(Q)$,
and let $\{\delta_1,\ldots,\delta_k\}$ be a set of representatives
of $N_T(Q)$-orbits on the $N_G(Q)$-orbit containing $\delta$.
We will construct an injection from $\LPS(Q,\delta)$ into
$\bigcup_{i=1}^k  {\rm L}^*_{\varphi_\mu}(Q,\delta_i)$.

To do this, we fix an arbitrary character $\chi\in \LPS(Q,\delta)$.
By Lemma \ref{CL3.1}(4), there exists an irreducible constituent $\theta$ of $\chi_N$ such that $\theta$ is $\pi$-factored with $\pi$-degree and $\theta^0=\mu$.
Applying the Clifford correspondence for $\pi$-induction (see Lemma \ref{4}),
we obtain a unique character $\psi\in\Irr(G_\theta|\theta)$ satisfying $\psi^{\pi G}=\chi$.
Clearly, the equality $\theta^0=\mu$ implies that $G_\theta\le G_\mu=T$,
and so we can write $\eta=\psi^{\pi T}$.
Then $\eta^{\pi G}=(\psi^{\pi T})^{\pi G}=\psi^{\pi G}=\chi$, and thus
$$(\eta^0)^G=((\delta_{(G,T)}\eta)^0)^G
=((\delta_{(G,T)}\eta)^G)^0=(\eta^{\pi G})^0=\chi^0=\varphi,$$
which forces $\eta^0\in\I(T)$.
On the other hand, since $N\le \Ker(\delta_{(T, G_\theta)})$ by Lemma \ref{1},
we know that $\delta_{(T, G_\theta)}\psi$ lies over $\theta$,
and hence $\eta=(\delta_{(T, G_\theta)}\psi)^T$ lies over $\theta$.
So $\eta^0$ lies over $\theta^0=\mu$, and by the Clifford correspondence for $\I$-character,
we conclude that $\eta^0=\varphi_\mu$.

Furthermore, since $\chi$ is a good lift for $\varphi$,
it is easy to see that $\eta$ is also a good lift for $\varphi_\mu$.
We may therefore assume that $\eta\in {\rm L}^*_{\varphi_\mu}(Q,\delta')$
for some linear character $\delta'$ of $Q$.
Since $\eta^{\pi G}=\chi$,
it follows that $(Q,\delta')$ is also a $*$-vertex for $\chi$.
Now Theorem \ref{unique} implies that $(Q,\delta')=(Q,\delta)^g$ for some element $g\in G$,
so $g\in N_G(Q)$. From this we conclude that $\eta$ has a $*$-vertex that is $N_T(Q)$-conjugate to one of the $(Q,\delta_i)$ for some $i\in\{1,\ldots, k\}$. Of course this pair $(Q,\delta_i)$ is itself a $*$-vertex for $\eta$, and we obtain
$\eta\in\bigcup_{i=1}^k  {\rm L}^*_{\varphi_\mu}(Q,\delta_i)$.

We claim that $\eta$ is determined uniquely by $\chi$.
To see this, it suffices to show that $\eta$ is independent of the choice of $\theta$.
Assume that $\theta'\in\Irr(N)$ is a constituent of $\chi_N$ with $(\theta')^0=\mu$.
Then $\theta'=\theta^x$ for some $x\in G$, and hence
$\mu=(\theta')^0=(\theta^x)^0=(\theta^0)^x=\mu^x$.
It follows that $x\in G_\mu=T$.
Let $\psi'\in\Irr(G_{\theta'}|\theta')$ be such that $(\psi')^{\pi G}=\chi$,
and let $\eta'=(\psi')^{\pi T}$.
Then $\psi'=\psi^x$, and thus $\eta'=\eta^x=\eta$, as desired.

Now we have established a well-defined map
$\chi\mapsto\eta$ from $\LPS(Q,\delta)$ into
$\bigcup_{i=1}^k  {\rm L}^*_{\varphi_\mu}(Q,\delta_i)$.
Since $\eta^{\pi G}=\chi$, this map is injective, and hence we have
$$|\LPS(Q,\delta)|\leq\sum_{i=1}^k|{\rm L}^*_{\varphi_\mu}(Q,\delta_i)|.$$
If $T<G$, then the inductive hypothesis yields
$$|{\rm L}^*_{\varphi_\mu}(Q,\delta_i)| \le |N_T(Q):N_T(Q,\delta_i)|.$$
Observe that $|N_T(Q):N_T(Q,\delta_i)|$ is the orbit size of the $N_T(Q)$-orbit containing $\delta_i$ under the action of $N_T(Q)$ on the $N_G(Q)$-orbit containing $\delta$, and thus
$$|\LPS(Q,\delta)| \le \sum_{i=1}^k|N_T(Q):N_T(Q,\delta_i)|
=|N_G(Q):N_G(Q,\delta)|.$$
The result follows in this case, and we may therefore assume that $G_\mu=G$.
That is, $\mu$ is invariant in $G$, so that $\varphi_N=e\mu$ for some integer $e$.
\end{proof}

\HJ{Step 2}
There exists a character $\theta\in\Irr(N)$ such that
$\theta$ is $N_G(Q,\delta)$-invariant and lies under every character in $\LPS(Q,\delta)$.
Furthermore, writing $V=G_\theta$, we have
$$|\LPS(Q,\delta)|\leq |\VI(V|Q)||N_V(Q):N_V(Q,\delta)|,$$
where $\VI(V|Q)$ is the set of those members of $\I(V|Q)$ that induce $\varphi$,
as in Lemma \ref{L2010-2.1}.
\EHJ
\begin{proof}
By Lemma \ref{CL3.1}, we set $\theta=\alpha\beta$,
where $\alpha\in\Irr(N)$ is the unique $\pi$-special character such that $\alpha^0=\mu$
and $\beta\in\Irr(N)$ is the unique $\pi'$-special character with
$\beta_{Q\cap N}=\delta_{Q\cap N}$.
Then $\theta$ is irreducible and $\theta^0=\alpha^0=\mu$.
Furthermore, since $\mu$ is invariant in $G$ by Step 1, it follows that $\alpha$ is $G$-invariant,
and in particular, we have $G_\beta=G_\theta=V$.
Observe that $N_G(Q,\delta)$ fixes $\delta_{Q\cap N}$, so it leaves $\beta$ unchanged.
This shows that $\theta$ is $N_G(Q,\delta)$-invariant, i.e. $N_G(Q,\delta)\le G_\theta$.
By Lemma \ref{CL3.1} again, it is easy to see that
every character in $\LPS(Q,\delta)$ lies over $\theta$,
that is, $\LPS(Q,\delta)\subseteq\Irr(G|\theta)$.

Assume first that $V=G$. Then $\theta$ is $G$-invariant,
and by Lemma \ref{N-nuc}, we have $N=G$.
It follows that $\LPS(Q,\delta)=\{\theta\}$,
and the result trivially holds in this case.

We can therefore assume that $V<G$.
For any character $\chi\in\LPS(Q,\delta)$,
there exists a unique $\psi\in\Irr(V|\theta)$ such that $\psi^{\pi G}=\chi$
by the Clifford correspondence for $\pi$-induction (Lemma \ref{4}),
so that $\chi\mapsto\psi$ defines an injection from
$\LPS(Q,\delta)$ into $\Irr(V|\theta)$.
We will show that $\psi^0\in\VI(V|Q)$.

First, we note that
$$\varphi=\chi^0=(\psi^{\pi G})^0=((\delta_{(G,V)}\psi)^G)^0
=((\delta_{(G,V)}\psi)^0)^G=(\psi^0)^G,$$
so $\psi^0$ is irreducible and induces $\varphi$.

Next, we claim that $(Q,\delta)$ is a $*$-vertex for $\psi$.
To see this, let $(Q^*,\delta^*)$ be a $*$-vertex for $\psi$.
Since $\theta$ is the unique irreducible constituent of $\psi_N$,
we have $\beta_{Q^*\cap N}=\delta^*_{Q^*\cap N}$ by Lemma \ref{CL3.1}.
Also, since $(Q^*,\delta^*)$ is clearly a $*$-vertex for $\chi$,
there exists some element $y\in G$ such that $(Q^*,\delta^*)=(Q,\delta)^y$ by Theorem \ref{unique}.  Now we have
$$\beta_{Q^y\cap N}^y=(\beta_{Q\cap N})^y=(\delta_{Q\cap N})^y=\delta^y_{Q^y\cap N}=\delta_{Q^*\cap N}^*=\beta_{Q^*\cap N}=\beta_{Q^y\cap N}.$$
Since $Q^y\cap N=(Q\cap N)^y$ is a $\Hall$ $\pi'$-subgroup of $N$ by Lemma \ref{CL3.1} again, it follows from Lemma \ref{res} that $\beta^y=\beta$,
and thus $y\in G_\beta=V$. This shows that $(Q,\delta)$ is also a $*$-vertex for $\psi$,
as claimed.

Finally, by Lemma \ref{vertex}, we see that $\psi^0\in \I(V)$ has vertex $Q$.
Since we have established that $\psi^0$ induces $\varphi$,
it follows that $\psi^0\in\VI(V|Q)$, as wanted.

Now let $\VI(V|Q)=\{\zeta_1,\ldots,\zeta_m\}$.
Then $\psi^0=\zeta_i$ for some $i\in\{1,\ldots,m\}$,
so that $\psi\in {\rm L}^*_{\zeta_i}(Q,\delta)$.
It follows that the map $\chi\mapsto\psi$ defines an injection
from $\LPS(Q,\delta)$ into $\bigcup_{i=1}^m {\rm L}^*_{\zeta_i}(Q,\delta)$,
and we have
$$|\LPS(Q,\delta)|\leq \sum_{i=1}^m |{\rm L}^*_{\zeta_i}(Q,\delta)|.$$
Since $V<G$, the inductive hypothesis guarantees that
$|{\rm L}^*_{\zeta_i}(Q,\delta)|\leq |N_V(Q):N_V(Q,\delta)|$,
and hence $|\LPS(Q,\delta)|\leq m|N_V(Q):N_V(Q,\delta)|= |\VI(V|Q)||N_V(Q):N_V(Q,\delta)|$.
\end{proof}

We can now complete the proof of the theorem.
By Step 2, we know that $N_G(Q,\delta)\le V$, so $N_V(Q,\delta)=N_G(Q,\delta)$.
Applying Corollary \ref{L2010-2.2}, we have $|\VI(V|Q)|\leq |N_G(Q):N_V(Q)|$.
This, together with the inequality in Step 2, yields
$$|\LPS(Q,\delta)|\leq |N_G(Q):N_V(Q)||N_V(Q):N_G(Q,\delta)|
=|N_G(Q):N_G(Q,\delta)|,$$
and the proof is now complete.
\end{proof}

As before, we prove a more general result, which covers Theorem B
by taking $\pi=\{2\}'$.

\begin{theorem}\label{B}
Let $G$ be a $\pi$-separable group with $2\notin\pi$,
and suppose that $\varphi\in\I(G)$ has vertex $Q$, where $Q$ is a $\pi'$-subgroup of $G$.
If $\delta\in\Lin(Q)$, then
$$|\LN(Q,\delta)|\leq|N_G(Q):N_G(Q,\delta)|,$$
and thus $|\LN|\le |Q:Q'|$.
\end{theorem}
\begin{proof}
Fix $\delta\in\Lin(Q)$. By Lemma \ref{unique}, we have $\LN(Q,\delta)=\LPS(Q,\delta_{(G,Q)}\delta)$, and
since $N_G(Q)$ fixes $\delta_{(G,Q)}$, it follows by Theorem \ref{B'} that
$$|\LPS(Q,\delta_{(G,Q)}\delta)|\le |N_G(Q):N_G(Q,\delta_{(G,Q)}\delta)|
=|N_G(Q):N_G(Q,\delta)|.$$
This proves the first assertion,
which implies that $|\LN|\le |Q:Q'|$ (by the remarks preceding the statement of Theorem B in the introduction).
\end{proof}

Finally, we prove Corollary C from the introduction.
Recall that a finite group $G$ is said to be an {\bf $M$-group}
if every $\chi\in\Irr(G)$ is induced from a linear character of some subgroup of $G$.
By Theorem 6.23 of \cite{I1976},
if $G$ is a solvable group with a normal subgroup $N$,
such that all Sylow subgroups of $N$ are abelian and $G/N$ is supersolvable,
then each subgroup of $G$ is an $M$-group.
We will prove the following somewhat more general result,
which clearly implies Corollary C.

\begin{cor}
Let $\varphi\in\I(G)$ with vertex $Q$, where $G$ is $\pi$-separable with $2\notin\pi$.
Suppose that there exists a normal subgroup $N$ of $G$, such that
$N$ has odd order and each subgroup of $G/N$ is an $M$-group. Then $\LN=\LP$, and thus $|\LP|\le |Q:Q'|$.
\end{cor}
\begin{proof}
Fix a character $\chi\in \LP$, and note that $N\le O_\pi(G)$ as $2\notin\pi$.
To prove that $\chi\in\LN$,
by Lemma \ref{lin-ver} it suffices to show that
each $\pi$-factored character $\psi\in\Irr(U)$ inducing $\chi$
has odd degree, where $N\le U\le G$ and $\psi_{\pi'}$ is primitive.
Let $\beta=\psi_{\pi'}$. Then, by definition, $N\le \Ker\beta$,
and we can view $\beta$ as a character of $U/N$.
By assumption, however, $U/N$ is an $M$-group,
so the primitive character $\beta$ must be linear.
It follows that $\psi$ has $\pi$-degree, and hence $\psi(1)$ is an odd number.
Finally, we have $|\LP|\le |Q:Q'|$ by Theorem \ref{B}.
\end{proof}

\section*{Acknowledgements}
This work was supported by the NSF of China (12171289) and the NSF of Shanxi Province
(20210302123429 and 20210302124077).



\begin{thebibliography}{19}
\bibitem{C2006} J.P. Cossey, A construction of two distinct canonical sets of lifts of Brauer characters of a $p$-solvable group, Arch. Math. {\bf 87} (2006) 385-389.

\bibitem{C2007} J.P. Cossey, Bounds on  the number of lifts of a Brauer character in a $p$-solvable group, J. Algebra {\bf 312} (2007) 699-708.

\bibitem{C2008} J.P. Cossey, Vertices of $\pi$-irreducible characters of groups of odd order,
Comm. Algebra  {\bf 36} (2008) 3972-3979.

\bibitem{C2009} J.P. Cossey, Vertices and normal subgroups in solvable groups,
J. Algebra  {\bf 321} (2009) 2962-2969.

\bibitem{C2010} J.P. Cossey, Vertex subgroups and vertex pairs in solvable groups,
in \emph{Character theory of Finite groups}, Contemp. Math.,
AMS, Prov. RI, {\bf 524} (2010) 17-32.

\bibitem{CL2010a} J.P. Cossey, M.L. Mark, Lifts of partial characters with cyclic defect groups,
J. Aust. Math. Soc. {\bf 89} (2010) 145-163.

\bibitem{CL2010} J.P. Cossey, M.L. Lewis, Counting lifts of Brauer characters,
arXiv:1008.1633 [math.GR], 2010.

\bibitem{CL2011} J.P. Cossey, M.L. Lewis, Inductive pairs and lifts in solvable groups,
J. Group Theory {\bf 14} (2011) 201-212.

\bibitem{CLN2011} J.P. Cossey, M.L. Lewis, G. Navarro,
The number of lifts of a Brauer character with a normal vertex,
J. Algebra  {\bf 328} (2011) 484-487.

\bibitem{C2012} J.P. Cossey, Vertex pairs and normal subgroups in groups of odd order,
Rocky Mt. J. Math. {\bf 42} (2012) 59-69.

\bibitem{CL2012} J.P. Cossey, M.L. Lewis, Lifts and vertex pairs in solvable groups,
Proc. Edinburgh Math. Soc. {\bf 55} (2012) 143-153.

\bibitem{I1986}I.M. Isaacs, Induction and restriction of $\pi$-special characters, Can. J. Math. {\bf 38} (1986) 576-604.

\bibitem{I1976} I.M. Isaacs, Character Theory of Finite Groups,
Providence, RI: AMS Chelsea Publishing, 2006.

\bibitem{I2018} I.M. Isaacs, Characters of Solvable Groups, RI: Amer. Math. Soc., 2018.

\bibitem{L2006} M.L. Lewis, Obtaining nuclei from chains of normal subgroups,
    J. Alg. Appl. {\bf 5} (2006) 215-229.

\bibitem{L2010} M.L. Lewis, Inducing $\pi$-partial characters with a given vertex, Ischina group theory 2010, 215-222, World Sci. Publ., Hackensack, NJ, 2012.

\bibitem{N1998} G. Navarro, Characters and Blocks of Finite Groups,
Cambridge: Cambridge University Press, 1998.

\bibitem{N2002} G. Navarro, Vertices for characters of $p$-solvable groups, Trans. Amer. Math. Soc. {\bf354} (2002) 2759-2773.

\bibitem{N2023} G. Navarro, Problems on characters: solvable groups, Publ. Mat. {\bf 67} (2023) 173-198.

\bibitem{WJ2022} L. Wang, P. Jin, The uniqueness of vertex pairs in $\pi$-separable groups,
Comm. Algebra (published online); see also arXiv:2212.04846 [math.GR].

\bibitem{WJ2023} L. Wang, P. Jin, Navarro vertices and lifts in solvable groups, arXiv:2302.09698 [math.GR], 2023.

\end{thebibliography}
\end{document}